\documentclass[11pt]{amsart}

\author{Ted Dobson}
\address{Ted Dobson, University of Primorska, UP IAM, Muzejski trg 2, SI-6000 Koper, Slovenia and\\ University of Primorska, UP FAMNIT, Glagolja\v{s}ka 8, SI-6000 Koper, Slovenia}
\email{ted.dobson@upr.si}
\author{Ademir Hujdurovi\' c}
\address{Ademir Hujdurovi\' c, University of Primorska, UP IAM, Muzejski trg 2, SI-6000 Koper, Slovenia and\\ University of Primorska, UP FAMNIT, Glagolja\v{s}ka 8, SI-6000 Koper, Slovenia}
\email{ademir.hujdurovic@upr.si}
\author{Klavdija Kutnar}
\address{Klavdija Kutnar, University of Primorska, UP IAM, Muzejski trg 2, SI-6000 Koper, Slovenia and\\ University of Primorska, UP FAMNIT, Glagolja\v{s}ka 8, SI-6000 Koper, Slovenia}
\email{klavdija.kutnar@upr.si}
\author{Joy Morris}
\address{Joy Morris, Department of Mathematics and Computer Science, University of Lethbridge,
Lethbridge,\\ Alberta, T1K 3M4, Canada}
\email{joy.morris@uleth.ca}\thanks{This work is supported in part by the Natural Science and Engineering Research Council of Canada (grant RGPIN-2017-04905) and in part by the Slovenian Research Agency (research programs P1-0285 and P1-0404 and research projects J1-1691, J1-1694, J1-1695, J1-1715, J1-9108, J1-9110, J1-9186,  N1-0062 and N1-0102).}
\usepackage{amssymb}

\newtheorem{thrm}{Theorem}[section]

\newtheorem{prop}[thrm]{Proposition}
\newtheorem{defin}[thrm]{Definition}
\newtheorem{lem}[thrm]{Lemma}

\newtheorem{hey}[thrm]{Remark}

\usepackage{tikz}
\usepackage{cite}
\usepackage{fullpage}
\usepackage{color}
\usepackage{pdfsync}
\usepackage{pdflscape}

\def\Z{{\mathbb Z}}
\def\Aut{{\rm Aut}}
\def\la{{\langle}}
\def\ra{{\rangle}}
\def\cal{\mathcal}
\def\gcd{{\rm gcd}}
\def\soc{{\rm soc}}
\def\fix{{\rm fix}}

\def\C{{\rm Z}}

\def\Stab{{\rm Stab}}

\def\tl{{\triangleleft}}

\def\mod{{\rm mod\ }}

\def\AGL{{\rm AGL}}

\def\PSL{{\rm PSL}}
\def\PGammaL{{\rm P\Gamma L}}

\def\SL{{\rm SL}}
\def\F{{\mathbb F}}

\def\PGL{{\rm PGL}}
\def\GL{{\rm GL}}
\def\PSp{{\rm PSp}}
\def\PG{{\rm PG}}
\def\GammaL{{\rm \Gamma L}}
\def\SigmaL{{\rm \Sigma L}}
\def\PSigmaL{{\rm P\Sigma L}}
\def\Cay{{\rm Cay}}

\usepackage{hyperref}

\begin{document}
\pagestyle{plain}

\title{Classification of vertex-transitive digraphs via automorphism group}

\bigskip

\maketitle

\begin{abstract}
In the mid-1990s, two groups of authors independently obtained classifications of vertex-transitive graphs whose order is a product of two distinct primes.  In the intervening years it has become clear that there is additional information concerning these graphs that would be useful, as well as making explicit the extensions of these results to digraphs.  Additionally, there are several small errors in some of the papers that were involved in this classification.  The purpose of this paper is to fill in the missing information as well as correct all known errors.
\end{abstract}

\section{Introduction}

The initial motivation for this paper came from some work \cite{DobsonHKM2017} done by the first four  authors that used a well-known classification of vertex-transitive graphs of order $pq$, where $p$ and $q$ are distinct primes.

The original classification had been obtained by two different groups of authors, each with their own perspective on what properties of these graphs were important.
One group (consisting of Maru\v si\v c and Scapellato) \cite{MarusicS1994} was primarily concerned with determining a minimal transitive subgroup of the automorphism group, while the other (consisting of Praeger, Xu, and several others) \cite{PraegerWX1993,PraegerX1993} was primarily concerned with determining the full automorphism groups of these graphs, and in particular determining all primitive or symmetric permutation groups that can act as the automorphism group of such a graph.  Although the results in the classifications are stated for graphs, the proofs as written apply equally to digraphs.

Over the years, it has become apparent that there are ``gaps" in the information about vertex-transitive digraphs of order $pq$ that are not addressed by either approach but would be useful to fill.  Specifically, the classification of vertex-transitive digraphs of order $pq$ that have imprimitive almost simple automorphism groups was incomplete, and the full automorphism group of the Maru\v si\v c-Scapellato (di)graphs was unknown.  Additionally, there are several  errors  in this classification, and these have propagated themselves in the literature.   The most significant of these errors, at least from the point of view of the difficulty in correcting the error, is with Praeger, Wang, and Xu's classification of symmetric Maru\v si\v c-Scapellato graphs \cite{PraegerWX1993}.

It is the purpose of this paper to fill the ``gaps" described in the preceding paragraph, and to correct the known errors. Finally, widespread reliance on results that contained errors has left a body of results that may or may not be correct; at best, the proofs need to be revised. We have not attempted to address all of these, but we provide a list of those that we are aware of.

\section{Preliminaries}

Throughout, $p$ and $q$ are distinct primes with $q<p$.  We begin with basic definitions. In particular, we define the classes of graphs and digraphs that will appear in what follows (with the exception of the Maru\v si\v c-Scapellato digraphs, whose definition is best presented in a group-theoretic context and is therefore postponed to Definition \ref{MSdigraphs}).  We denote the arc-set of a digraph $\Gamma$ by $A(\Gamma)$.  The most commonly studied class of vertex-transitive digraphs are  Cayley digraphs.

\begin{defin}
Let $G$ be a group and $S\subseteq G$. Define the {\bf \mathversion{bold} Cayley digraph of
$G$ with connection set $S$}, denoted $\Cay(G,S)$, to be the digraph with $V(\Cay(G,S)) = G$ and $A(\Cay(G,S)) = \{(g,gs):g\in G, s\in S\}$.
\end{defin}

Note that we use the term digraphs to include graphs. If $\Gamma$ is a digraph satisfying $(x,y)\in A(\Gamma)$ if and only if $(y,x)\in A(\Gamma)$, then we will say that $\Gamma$ is a {\bf graph}, and replace each pair $(x,y)$ and $(y,x)$ of symmetric ordered pairs in $A(\Gamma)$ by the unordered pair $\{x,y\}$ in the \textbf{edge set} $E(\Gamma)$, which takes the place of the arc set.  The next-most-commonly-encountered class of vertex-transitive digraphs are metacirculant digraphs, first defined by Alspach and Parsons \cite{AlspachP1982} (although they only defined metacirculant graphs).

\begin{defin}\label{metacirculantdefin}
Let $V = \Z_m\times\Z_n$, $\alpha\in\Z_n^*$, and $S_0,\ldots,S_{m-1}\subseteq\Z_n$ such that $\alpha^mS_i = S_i$, $i\in\Z_n$.  Define an {\bf \mathversion{bold}$(m,n,\alpha,S_0,\ldots,S_{m-1})$-metacirculant digraph} $\Gamma = \Gamma(m,n,\alpha,S_0,\ldots,S_{m-1})$ by $V(\Gamma) = \Z_m\times\Z_n$ and $A(\Gamma) = \{(\ell,j),(\ell + i,k)):k-j\in \alpha^\ell S_i\}$.
We also define an {\bf \mathversion{bold} $(m,n)$-metacirculant digraph} to be a digraph that is an $(m,n,\alpha,S_0,\ldots,S_{m-1})$-metacirculant digraph for some $\alpha$ and some $S_0, \ldots, S_{m-1}$ as above.
\end{defin}

Many Cayley digraphs and metacirculant digraphs have the important property of imprimitivity that assists in any effort to understand their automorphisms.

\begin{defin}\label{imprimitive}
Let $G\le S_X$ be transitive.  A subset $B\subseteq X$ is a {\bf block} of $G$ if whenever $g\in G$, then $g(B)\cap B = \emptyset$ or $B$.  For a block $B$ of $G$, the set $\mathcal{B}=\{g(B)\mid g \in G\}$ is called a {\bf $G$-invariant partition}.  If $B = \{x\}$ for some $x\in X$ or $B = X$, then $B$ is a {\bf trivial block}.  Any other block is nontrivial.  If $G$ has a nontrivial block, then $G$ is {\bf imprimitive}. If $G$ is not imprimitive, we say $G$ is {\bf primitive}.

If $\Gamma$ is a digraph, then we say that $\Gamma$ {\bf admits an imprimitive action} if there is some transitive group $G \le \Aut(\Gamma)$ that is imprimitive.  We say that $\Gamma$ {\bf admits no imprimitive action} if every transitive group $G \le \Aut(\Gamma)$ is primitive. We say that $\Gamma$ is
 {\bf primitive} if $\Aut(\Gamma)$ is primitive, and $\Gamma$ is {\bf imprimitive} if $\Aut(\Gamma)$ is imprimitive. We refer to any block of $\Aut(\Gamma)$ as a block of $\Gamma$ also.
\end{defin}

It is important for us to make these definitions about $\Gamma$. One of the sources of confusion in the literature is that Maru\v si\v c and Scapellato referred to a digraph as $m$-imprimitive whenever it admits an imprimitive action with blocks of size $m$, even if the full automorphism group is primitive.

We observe that a digraph $\Gamma$ of order $pq$ must lie in one of three families: $\Gamma$ is primitive; $\Gamma$ is imprimitive with blocks of size $p$; or $\Gamma$ is imprimitive with blocks of size $q$. Note that the second and third families are not mutually exclusive.

Maru\v si\v c provided some of the early analysis of vertex-transitive graphs of order $pq$.

\begin{prop}[Proposition 3.3, \cite{Marusic1988}] \label{prop3.3Dragan}
The graphs of order $pq$ that admit an imprimitive action with blocks of size $p$ are precisely the $(q,p)$-metacirculant graphs.
\end{prop}

\begin{thrm}[Theorem 3.4, \cite{Marusic1988}]\label{thm3.4Dragan}
Let $\Gamma$ be a graph of order $pq$ that admits an imprimitive  action of the group $G$ with a $G$-invariant partition ${\cal B}$. Suppose that $\Gamma$ is not a metacirculant graph. Then the kernel of the action of $G$ on $\mathcal B$ is trivial, and $G$ is nonsolvable.
\end{thrm}

The proofs of both of these results as written apply equally to digraphs.

In later work with Scapellato, he extended these results to show the following.

\begin{thrm}[Theorem, \cite{MarusicS1992}]\label{thm:MS1992}
Let $\Gamma$ be a vertex-transitive digraph of order $pq$ that admits an imprimitive action but is not metacirculant. Then every (transitive) imprimitive subgroup of $\Aut(\Gamma)$ admits blocks of size $q$; $p=2^{2^a}+1$ is a Fermat prime, $q$ divides $p-2$, and $\Gamma$ is a Maru\v si\v c-Scapellato graph (see Definition~\ref{MSdigraphs}).
\end{thrm}

The classification of groups of automorphisms as ``primitive" or ``imprimitive" is a natural one. Observe that a primitive group $G$ cannot contain an intransitive normal subgroup, because the orbits of such a group would give rise to a $G$-inviariant partition \cite[Proposition 7.1]{Wielandt1964}. However, $G$-invariant partitions can also arise even if $G$ has no intransitive normal subgroup.
The following definition was first introduced by Praeger.

\begin{defin}
A transitive group is called {\bf quasiprimitive} if every nontrivial normal subgroup is transitive.
\end{defin}

As we have just observed, every primitive group is quasiprimitive, and quasiprimitive groups are a generalization of primitive groups.

Vertex-transitive digraphs with quasimprimitive automorphism groups are usually studied via their orbital digraphs, which we now define.

\begin{defin}
Let $G$ act on $X\times X$ in the canonical way, that is $g(x,y) = (g(x),g(y))$. The orbits of this action are called {\bf orbitals}.  One orbital is the diagonal, or $\{(x,x):x\in X\}$, and is called the {\bf trivial orbital}. We assume here that ${\cal O}_1,\ldots{\cal O}_r$ are the nontrivial orbitals.  Define digraphs $\Gamma_1,\ldots,\Gamma_r$ by $V(\Gamma_i) = X$ and $E(\Gamma_i) = {\cal O}_i$.  The set $\{\Gamma_i:1 \le i \le r\}$ is the set of {\bf \mathversion{bold} orbital digraphs of $G$}.  A {\bf \mathversion{bold} generalized orbital digraph of $G$} is an arc-disjoint union of some orbital digraphs of $G$ (that is, identify vertices in the natural way amongst a set of orbital digraphs, and take the new arc set to be the union of the arcs that are in any of the orbital digraphs).  We say an orbital is {\bf self-paired} if the corresponding orbital digraph is a graph.
\end{defin}

Orbital digraphs of a group $G$ are often given in terms of their suborbits.

\begin{defin}
Let $G\le S_n$ be transitive and $x$ a point.  The orbits of $\Stab_G(x)$ are the {\bf \mathversion{bold} suborbits of $G$ with respect to $x$}.
\end{defin}

Notice that in an orbital digraph of $G$, the outneighbors of $x$ and inneighbors of $x$ are both suborbits of $G$ with respect to $x$.  We finish this section with group- and graph-theoretic terms that relate to graph quotients.

\begin{defin}
Suppose $G\le S_n$ is a transitive group that has a $G$-invariant partition ${\cal B}$ consisting of $m$ blocks of size $k$.  Then $G$ has an {\bf \mathversion{bold} induced action on ${\cal B}$}, denoted $G/{\cal B}$.  Namely, for $g\in G$, define $g/{\cal B}:{\cal B}\mapsto{\cal B}$ by $g/{\cal B}(B) = B'$ if and only if $g(B) = B'$, and set $G/{\cal B} = \{g/{\cal B}:g\in G\}$.  We also define the {\bf \mathversion{bold} fixer of ${\cal B}$ in $G$}, denoted $\fix_G({\cal B})$, to be $\{g\in G:g/{\cal B} = 1\}$.  That is, $\fix_G({\cal B})$ is the subgroup of $G$ which fixes each block of ${\cal B}$ set-wise.
\end{defin}

Observe that $\fix_G({\cal B})$ is the kernel of the induced homomorphism $G\mapsto S_{\cal B}$ that arose previously in the statement of Theorem~\ref{thm3.4Dragan}, and as such is normal in $G$.

\begin{defin}\label{block quotient def}
Let $\Gamma$ be a vertex-transitive digraph that admits an imprimitive action of the group $G$ with a $G$-invariant partition ${\cal B}$.  Define the {\bf \mathversion{bold} block quotient digraph of $\Gamma$ with respect to ${\cal B}$}, denoted $\Gamma/{\cal B}$, to be the digraph with vertex set ${\cal B}$ and arc set $$\{(B,B'):B\not = B'\in{\cal B}\textit{ and }(u,v)\in A(\Gamma)\textit{ for some }u\in B\textit{ and } v \in B'\}.$$
\end{defin}

Note that $\Aut(\Gamma)/{\cal B}\le\Aut(\Gamma/{\cal B})$.

\section{Automorphism groups of $(q,p)$-metacirculant digraphs whose full automorphism group admits only blocks of size $q$}

Our original interest in this problem arose when we were studying a particular Cayley digraph of the nonabelian group of order $21$ whose automorphism group is a nonabelian simple group but is imprimitive.  This digraph is included in the Maru\v si\v c-Scapellato characterization as a metacirculant digraph as its automorphism group contains the nonabelian group of order $21$. It does not appear elsewhere in that characterization as Maru\v si\v c and Scapellato were interesteed in finding a minimal transitive subgroup (indeed, they define a primitive graph to be one in which {\it every} transitive subgroup of the automorphism group is primitive), and so they were not concerned with its full automorphism group.  This digraph does not occur in the Praeger-Xu characterization, as they were interested in graphs (and occasionally digraphs) whose full automorphism group is primitive (indeed, they define a primitive graph to be one in which the full automorphism group is primitive).  So in neither characterization of vertex-transitive graphs of order $pq$ were such digraphs looked for.  Finally, this digraph does not arise in \cite[Theorem 3.2(1)]{Dobson2006a} since that result only holds for graphs, not digraphs.  Thus there is a small gap in the literature here.



The aim of this section of our paper is to fill in this gap. Fortunately, the work by Maru\v si\v c and Scapellato \cite{MarusicS1992} can be easily modified to help in this goal. Indeed, Maru\v si\v c and Scapellato's work is actually stronger than advertised through the statement of their results, and an additional goal of this section is to make this stronger work more apparent, as from our work on this paper we believe that such stronger statements may be useful.  We note that when writing a wreath product, we use the convention that the first group written is acting on the partition, and the second is acting within each block. Some authors, including Praeger et al, use the opposite order.

\begin{thrm}\label{gendigraphauto}
Let $\Gamma$ be a vertex-transitive digraph of order $pq$, where $q < p$ are distinct primes such that $G\le \Aut(\Gamma)$ is quasiprimitive  and has a $G$-invariant partition ${\cal B}$ with blocks of size $q$.  Additionally, suppose that ${\cal B}$ is also an $\Aut(\Gamma)$-invariant partition.  Then $G$ is an almost simple group and one of the following is true:
\begin{enumerate}
\item $\Gamma$ is a nontrivial wreath product and $\Aut(\Gamma)$ contains $G/{\cal B}\wr(\Stab_G(B)\vert_B)$ which contains a regular cyclic subgroup $R$, where $B\in {\cal B}$, or \label{wreath}
\item $\Gamma$ is isomorphic to a generalized orbital digraph of $\PSL(2,11)$ that is not a generalized orbital digraph of $\PGL(2,11)$ of order $55$. Moreover, $\Gamma$ is a Cayley digraph of the nonabelian group of order $55$, and its full automorphism group is $\PSL(2,11)$, or\label{case:2-11}
\item $\Gamma$ is isomorphic to a generalized orbital digraph of $\PSL(3,2)$ of order $21$ that is not a generalized orbital digraph of $\PGammaL(3,2)$. Moreover, $\Gamma$ is a Cayley digraph of the nonabelian group of order $21$, and its full automorphism group is $\PSL(3,2)$, or \label{case:3-2}
\item $\Gamma$ is not metacirculant; $p=2^{2^s}+1$ is a Fermat prime, and $q$ divides $p-2$. Further, the minimal transitive subgroup $G$ of $\Aut(\Gamma)$ that admits only a $G$-invariant system of $p$ blocks of size $q$ is isomorphic to $\SL(2,2^s)$, and $\Aut(\Gamma)$ is isomorphic to a subgroup of $\Aut(\SL(2,2^s))$. \label{MS-graphs}
\end{enumerate}
\end{thrm}

\begin{proof}
Almost all of the proof is contained in \cite{MarusicS1992}.  We analyze the digraph structures essentially as they do in the proof of their main theorem.

Since $G$ is quasiprimitive, it has no nontrivial intransitive normal subgroups. So $\fix_G({\cal B}) = 1$ and $G/{\cal B}\cong G$ is of prime degree $p$.  As $G$ does not have a normal Sylow $p$-subgroup, neither does $G/{\cal B}$, and so by Burnside's Theorem \cite[Corollary 3.5B]{DixonM1996} $G/{\cal B}$ is doubly-transitive, and by another theorem of Burnside \cite[Theorem 4.1B]{DixonM1996}, $G/{\cal B}$ has nonabelian simple socle.  Consequently, $G$ is nonsolvable and $G/{\cal B}\cong G$ is almost simple.

The $2$-transitive groups of prime degree are known (they are given for example in \cite[Proposition 2.4]{MarusicS1992}).  The various cases, with the one exception of $\PSL(2,2^k)$, are then analyzed in \cite{MarusicS1992}. They are almost all either rejected as impossible using group theoretic arguments or \cite[Proposition 2.1]{MarusicS1992} (which is purely about the permutation group structure and also applies to our situation), or determined to be metacirculants using \cite[Proposition 2.2]{MarusicS1992}, which is almost sufficient for our purposes. Maru\v{s}i\v{c} and Scapellato in fact showed that whenever $G$ is a group satisfying the hypothesis of \cite[Proposition 2.2]{MarusicS1992}, and $\Gamma$ is a digraph (they only considered graphs but their proof works for digraphs) with $G\le\Aut(\Gamma)$, then either $\Gamma$ or its complement is disconnected.  This implies that $\Aut(\Gamma)$ is a wreath product, and $\Aut(\Gamma)$ contains $G/{\cal B}\wr(\Stab_G(B)\vert_B)$ which contains a regular cyclic subgroup $R$, where $B\in {\cal B}$, which is what we need here.  There are two possible group structures for $G$ that do not succumb to this general approach, and \cite{MarusicS1992} use direct arguments to show that the corresponding (di)graphs are metacirculant. We need to address these exceptional possibilities separately.

The first exception occurs in the proof of \cite[Proposition 2.7]{MarusicS1992} when handling the case $G = \PSL(2,11)$ of degree $55$.  In this case it is argued that $\PSL(2,11)$ contains a regular metacyclic subgroup that has blocks of size $11$. This is a contradiction to the hypothesis of \cite[Proposition 2.7]{MarusicS1992}, so finishes the argument for them; for us, it shows that these digraphs are Cayley digraphs (as claimed), and $(q,p)$-metacirculants.  It can be verified in \texttt{magma} \cite{magma} that the only regular subgroup of $\PSL(2,11)$ in its action on $55$ points  is the nonabelian group of order $55$. Since $\PGL(2,11)$ is primitive, the digraphs that arise in this case are precisely those whose full automorphism group is $\PSL(2,11)$.

The second exception occurs at the beginning of \cite[Proposition 3.5]{MarusicS1992}, namely when $G = \PSL(3,2)$ and $\Gamma$ is of order $21$.  Here, Maru\v si\v c and Scapellato note that $\PSL(3,2)$ in its action on $21$ points has a (transitive) nonabelian subgroup of order $21$, and so $\Gamma$ is a metacirculant (which is enough for their purposes, and for us again shows that $\Gamma$ is a Cayley graph on the nonabelian group of order $21$).  By the Atlas of Finite Simple Groups the group $\PSL(3,2)$ in its representation on $21$ points has suborbits of length $1$, $2^2$, $4^2$, and $8$, with the suborbits of lengths $4$ being non self-paired. The action of $\PGammaL(3,2)$ is primitive, so again orbital digraphs of that group do not meet our hypotheses and we are interested only in those digraphs whose full automorphism group is $\PSL(3,2)$.

Finally, the case where $G$ has socle $\PSL(2,2^k) = \SL(2,2^k)$ is mainly analyzed in \cite{MarusicS1994}, where, for example, the orbital digraphs of the groups are determined.  In \cite[Theorem]{MarusicS1992}, they show if an imprimitive representation of $\SL(2,2^k)$ has order $qp$ and is contained in the automorphism group of a metacirculant digraph $\Gamma$ of order $qp$, then either it either contains the complete $p$-partite graph where each partition has size $q$ (and are the blocks of ${\cal B}$), or is contained in the complement of this graph.  These digraphs are easily seen to be circulant as either $\Gamma$ or its complement is again disconnected.  The arithmetic conditions are also derived there.
\end{proof}

From a closer analysis of the suborbits of $\PSL(2,11)$ and of $\PSL(3,2)$, we can derive additional information about the digraphs that arise in this analysis. For $\PSL(3,2)$, we use \texttt{magma} for this analysis.
The orbital digraphs of $\PSL(2,11)$ are examined in \cite[Example 2.1]{LuX2003}.
The suborbits are of length $1,4,4,4,6,12,12,$ and $12$.
Two suborbits of length $12$ are the only ones that are not self-paired, and the corresponding orbital digraphs have automorphism group $\PSL(2,11)$ which is imprimitive (as $\PSL(2,11)$ has disconnected orbital digraphs).
Thus, a generalised orbital digraph that is not a graph must use exactly one of these.
Two suborbits of length $4$ have disconnected orbital graphs and their union is an orbital graph of $\PGL(2,11)$, while all of the other suborbits are also suborbits of $\PGL(2,11)$.
Thus, in order to avoid $\PGL(2,11)$ in the automorphism group of an orbital graph, we must include exactly one of these. We summarize this extra information in the following remark.

\begin{hey}
If $\Gamma$ arises in Theorem~\ref{gendigraphauto}(\ref{case:2-11}) and is a graph, then it has a subgraph of valency $4$ that is a disconnected orbital graph of $\PSL(2,11)$, and the other disconnected orbital graph of $\PSL(2,11)$ (which is the image of this one under the action of $\PGL(2,11)$) is not a subgraph of $\Gamma$ (but $\Gamma$ itself is connected).

If $\Gamma$ arises in Theorem~\ref{gendigraphauto}(\ref{case:2-11}) and is not a graph, then it has a subdigraph of valency $12$ that is a non-self-paired orbital digraph of $\PSL(2,11)$, and whose paired orbital digraph of $\PSL(2,11)$ is not a subdigraph of $\Gamma$.

If $\Gamma$ arises in Theorem~\ref{gendigraphauto}(\ref{case:3-2}) then \texttt{magma} \cite{magma} has been used to verify that $\Gamma$ cannot be a graph. It has a subdigraph of valency $4$ that is a non-self-paired orbital digraph of $\PSL(3,2)$, and whose paired orbital digraph of $\PSL(3,2)$ is not a subdigraph of $\Gamma$.
\end{hey}

\begin{hey}
There are several instances, other than the complete graph and its complement, where a quasiprimitive group $G$ with nontrivial $G$-invariant partition ${\cal B}$, is contained in the full automorphism group of a digraph $\Gamma$ of order $qp$, but the automorphism group $\Aut(\Gamma)$ is primitive.  We list the exceptions or not in the same order as in Theorem \ref{gendigraphauto}:
\begin{enumerate}
\item There are no such cases if Theorem \ref{gendigraphauto} (\ref{wreath}) holds as $\Z_{qp}$ is a Burnside group \cite[Corollary 3.5A]{DixonM1996}.  This implies $\Aut(\Gamma)$ is doubly-transitive and so $\Aut(\Gamma)=S_{qp}$.
\item If Theorem \ref{gendigraphauto} (\ref{case:2-11}) holds then there are graphs whose automorphism group is primitive and equal to $\PGL(2,11)$ on $55$ points that contain the quasiprimitive and imprimitive representation of $\PSL(2,11)$ on $55$ points.  These graphs are explicitly described in \cite[Lemma 4.3]{PraegerX1993}.
\item  If Theorem \ref{gendigraphauto} (\ref{case:3-2}) holds, then there are graphs whose automorphism group is $\PGammaL(3,2)$ in its primitive representation on $21$ points that contains the quasiprimitive and imprimitive representation of $\PSL(3,2)$ on $21$ points.  These graphs are explicitly described in \cite[Example 2.3]{WangX1993}.
\item If Theorem \ref{gendigraphauto} (\ref{MS-graphs}) holds, then there are graphs whose automorphism group is primitive but contains the quasiprimitive and imprimitive representation of $\SL(2,2^{2^s})$ on $qp$ points.  These graphs are explicitly described in the proof of \cite[Theorem 2.1]{MarusicS1994}, starting in the last paragraph on page 192.
\end{enumerate}
\end{hey}

\section{Automorphism groups of Maru\v si\v c-Scapellato digraphs}

We turn now to the next ``gap" in information about vertex-transitive digraphs of order $pq$, where there is also an error.  The gap is that there is not an algorithm to calculate the full automorphism group of every vertex-transitive digraph of order $pq$.

The automorphism groups of circulant digraphs of order $pq$ are found in \cite{KlinP1981}.
One of the authors of this paper determined the automorphism groups of metacirculant graphs of order $pq$ that are not circulant \cite{Dobson2006a} and that argument works for digraphs as well provided that the full automorphism group is not an almost simple group (we dealt with this last possibility in the previous section). Praeger and Xu \cite{PraegerX1993} determined the full automorphism group of graphs of order $pq$ in every case where that group is acting primitively. In light of Theorem~\ref{thm:MS1992}, this means that the gap in the problem of determining the full automorphism group of vertex-transitive digraphs of order a product of two distinct primes reduces to determining the automorphism groups of imprimitive Maru\v si\v c-Scapellato digraphs of order $pq$ that are not metacirculant graphs, where $p$ is a Fermat prime, and $q$ divides $p-2$. In the process of filling this gap we will fix an error in \cite{PraegerX1993}. Unfortunately, there is also an error of omission in  \cite{PraegerX1993} that we will need to correct, but we leave this for Section~\ref{sec:prim-errors}.

For the remainder of this section, we may therefore assume that $p=2^{2^t}+1$ is a Fermat prime, and $q$ is a divisor of $p-2$ (so $t\ge1$). For convenience, we write $s=2^t$. This means we are considering a restricted subclass of Maru\v{s}i\v{c}-Scapellato digraphs, since the original definition allowed $p=2^s$ without any conditions on $s$.

The Maru\v si\v c-Scapellato digraphs are vertex-transitive digraphs that are generalized orbital digraphs of $\SL(2,2^s)$. They were first studied by Maru\v si\v c and Scapellato in \cite{MarusicS1992,MarusicS1993}.  Praeger, Wang, and Xu \cite{PraegerWX1993} determined the automorphism groups of Maru\v si\v c-Scapellato graphs of order $pq$ that are also symmetric (i.e.\ arc-transitive), partially filling the gap we are addressing here.  One of the authors of this paper studied the full automorphism groups of Maru\v si\v c-Scapellato graphs in \cite{Dobson2016} and was able to say a great deal about them, but left their complete determination as an open problem \cite[Problem 1]{Dobson2016}.

We now discuss the construction of Maru\v si\v c-Scapellato graphs, using a combination of the approaches followed in \cite{Dobson2016} and \cite{MarusicS1993}.

Let $I_2$ be the $2\times 2$ identity matrix, and set $Z = \{aI_2:a\in{\mathbb F}_{2^{s}}^*\}$, the set of all {\bf scalar matrices}.  The name $Z$ is chosen as $Z = \C(\GL(2,2^s))$, the center of $\GL(2,2^s)$. Let $\mathbb F_{2^s}^2$ denote the set of all $2$-dimensional vectors whose entries lie in $\mathbb F_{2^s}$. Clearly $\SL(2,2^s)$ is transitive on ${\mathbb F}_{2^s}^2 - \{(0,0)\}$.  It is also clear that $\SL(2,2^s)$ permutes the {\bf projective points} $\PG(1,2^s)$, where a projective point is the set of all vectors other than $(0,0)$ that lie on a line.  Notice that there are $2^s+1$ projective points, and $\PG(1,2^s)$ is an invariant partition of $\SL(2,2^s)$ in its action on ${\mathbb F}_{2^s}^2 - \{(0,0)\}$ with $2^s + 1$ blocks of size $2^s - 1$.  This action is faithful.  That is, $\SL(2,2^s)/\PG(1,2^s)\cong\SL(2,2^s)$, or equivalently, $\fix_{\SL(2,2^s)}(\PG(1,2^s)) = 1$.

It is traditional to identify the projective points with elements of ${\mathbb F}_{2^s}\cup\{\infty\}$  in the following way:  The nonzero vectors in the one-dimensional subspace generated by $(1,0)$, will be identified with $\infty$.  Any other one-dimensional subspace is generated by a vector of the form $(c,1)$, where $c\in{\mathbb F}_{2^s}$.  The nonzero vectors in the one-dimensional subspace generated by $(c,1)$ will be identified with $c$.

For $a\in{\mathbb F}_{2^s}^*$, let $\sqrt{a}$ be the unique element of ${\mathbb F}_{2^s}^*$ whose square is $a$, and

$$k_a = \left[\begin{array}{ll}
\sqrt{a} & 0\\
0 & \sqrt{a}^{-1}
\end{array}\right].$$

\noindent Set $K = \{k_a:a\in{\mathbb F}_{2^s}^*\}$.  It is clear that $k_a$ stabilizes the projective point $\infty$ and that for any generator $\omega$ of ${\mathbb F}^*_{2^s}$, $\langle k_\omega\rangle=K$ is cyclic (of order $2^s-1$), since $\sqrt{\omega}$ also generates $\mathbb F^*_{2^s}$.
Additionally, it is clear that every element of the set-wise stabilizer of $\infty$ in $\SL(2,2^k)$ has the same action on $\infty$ as some element of $K\vert_{\infty}$ (the entry in the top-right position is irrelevant to the action on $\infty$).
Let $J\le K$ be the unique subgroup of order $\ell$, where $\ell$ is a fixed divisor of $2^s - 1$ (under our assumptions, we will take $\ell=(2^s-1)/q$).  By \cite[Exercise 1.5.10]{DixonM1996}, every orbit of $J\vert_\infty$ is a block of $\SL(2,2^s)$, and so $\SL(2,2^s)$ has an invariant partition ${\cal D}_\ell$ with blocks of size $\ell$ (the blocks whose points lie ``within" the projective point $\infty$ of $\PG(1,2^s)$ -- that is, those blocks consisting of points whose second entry is $0$ -- are the orbits of $J\vert_\infty$, and the other blocks are the images of these orbits under $\SL(2,2^s)$). These blocks of $\mathcal D_\ell$ will be the vertices of the generalised orbital digraphs of $\SL(2,2^s)$, and it is the action of $\SL(2,2^s)$ on these blocks that produces the Maru\v si\v c-Scapellato digraphs. Under our assumptions, there are $pq$ blocks in $\mathcal D_\ell$.

As mentioned above, the blocks of $\mathcal D_\ell$ are the images of the orbits of $J\vert_{\infty}$ under the action of $\SL(2,2^s)$, so each lies within a point of $\PG(1,2^s)$; that is, ${\cal D}_\ell\preceq\PG(1,2^s)$.
Now $\SL(2,2^s)/{\cal D}_\ell$ is a faithful representation of $\SL(2,2^s)$ (as $\fix_{\SL(2,2^s)}(\PG(1,2^s)) = 1$ and ${\cal D}_\ell\preceq\PG(1,2^s)$).  Additionally, the $\SL(2,2^s)$-invariant partition $\PG(1,2^s))$ induces the invariant partition ${\cal B} = \PG(1,2^s)/{\cal D}_\ell$ of $\SL(2,2^s)/{\cal D}_\ell$, and $\mathcal B$ consists of $2^s + 1$ blocks whose size in general is $m = (2^s - 1)/\ell$ (under our assumptions, $m=q$).  We will use the notation ${\cal B} = \PG(1,2^s)/{\cal D}_\ell$ throughout this section.  The elements of $\mathcal B$ will be the blocks of our digraphs of order $pq$, and will have size $q$ (and there are $p$ of them), so for our purposes and henceforth in this section, we have $q=m=(2^s-1)/\ell$. It is shown in \cite{MarusicS1993} that ${\cal B}$ is the unique complete block system of $\SL(2,2^s)/{\cal D}_\ell$ with blocks of size $q$. The following result is \cite[Lemma 2.3]{MarusicS1993}.

\begin{lem}\label{orbitalint}
$\SL(2,2^s)/{\cal D}_\ell$ has $q$ suborbits of length $1$ and $q$ suborbits of length $2^s$.  Additionally, for a suborbit $S$ of length $2^s$, $\vert S\cap(c/{\cal D}_\ell)\vert = 1$ for every projective point $c\in\PG(1,2^s)$.
\end{lem}

Note that this implies that the valency of an orbital digraph of $\SL(2,2^s)$ is either $1$ or $2^s$.  Additionally, as $\SL(2,2^s)/\PG(1,2^s) = \PSL(2,2^s)$ is doubly-transitive, the previous result also implies that the orbital digraphs of of $\SL(2,2^s)/{\cal D}_\ell$ having valency $2^s$ are graphs.  We now define Maru\v si\v c-Scapellato digraphs, and the fact that some orbital digraphs are graphs and some are not will cause us to naturally define these digraphs in terms of the edges which are not arcs as well as arcs that need not be edges.

\begin{defin}\label{MSdigraphs}
Let $s$ be a positive integer, and $q$ a divisor of $2^s - 1$, $S\subset\Z_q^*$, $\emptyset \subseteq T\subseteq\Z_q$, and $\omega$ a primitive element of $\F_{2^s}$.  The digraph $X(2^s,q,S,T)$ has vertex set $\PG(1,2^s)\times\Z_q$.  The out-neighbors  of $(\infty,r)$ are $\{(\infty,r + a):a\in S\}$ while the neighbors of $(\infty,r)$ are $\{(y,r+b):y\in\F_{2^s},b \in T\}$.  The out-neighbors of $(x,r)$, $x\in\F_{2^s}$, are given by $\{(x,r + a):a\in S\}$ while the neighbors of $(x,r)$ are

$$\{(\infty,r-b):b \in T\}\cup\{(x + \omega^i,-r+b+2i):i\in\Z_{2^s - 1},b \in T\}.$$

\noindent The digraph $X(2^s,q,S,T)$ is a {\bf Maru\v si\v c-Scapellato digraph}.
\end{defin}

In \cite{MarusicS1993} Maru\v si\v c and Scapellato only defined graphs, but their definition, with the obvious modifications, also define digraphs as above - see \cite{MarusicS1994a}.  Additionally, they required that $\emptyset\subset T\subset\Z_q$ as they wished their family to be disjoint from other already known families of graphs.  If $\emptyset = T$ or $T = \Z_q$ then the resulting digraphs are either disconnected or complements of disconnected digraphs, and so have automorphism group either a nontrivial wreath product or a symmetric group.  They also showed that with their definition, Maru\v si\v c-Scapellato digraphs are isomorphic to some, but not all, generalized orbital digraphs of $\SL(2,2^s)/{\cal D}_\ell$.  We prefer the more general definition that includes all generalized orbital digraphs of $\SL(2,2^s)/{\cal D}_\ell$.
However, the distinction Maru\v si\v c and Scapellato made is also important, so if $T = \emptyset$ or $\Z_q$, we will call such a Maru\v si\v c-Scapellato digraph a {\bf degenerate Maru\v si\v c-Scapellato digraph}.

We have said that the vertices of the Maru\v si\v c-Scapellato digraphs are the blocks of $\mathcal D_\ell$. The blocks of $\mathcal D_\ell$ are two-dimensional vectors that are subsets of projective points; in fact, it may be useful to the reader if we describe the blocks of $\mathcal D_\ell$ more precisely here. We assume that a primitive element $\omega$ of $\F_{2^s}$ has been chosen and is fixed. Then each block $D \in \mathcal D_\ell$ has one of the following forms:
\begin{eqnarray*}
\{(\sqrt{\omega}^{qj+r},0): 0 \le j \le \ell-1\} & &\text{(in the projective point $\infty$)}\\
\{(\sqrt{\omega}^{qj+c+r},\sqrt{\omega}^{qj+m}):0 \le j \le \ell-1\}&&\text{(in the projective point $\sqrt{\omega}^c$)}\\
\{(0,\sqrt{\omega}^{qj+r}): 0 \le j \le \ell-1\}&&\text{(in the projective point $0$),}
\end{eqnarray*}
for some fixed $0 \le r \le q-1$. The action of any element of $\SL(2,2^s)$ on any one of these sets is easy to calculate. Clearly, the definition that we have given for the Maru\v si\v c-Scapellato graphs does not have these sets as vertices; its vertices are the elements of $\PG(1,2^s) \times \Z_q$.

In \cite[Theorem 3.1]{MarusicS1993}, Maru\v si\v c and Scapellato show that the imprimitive orbital digraphs of $\SL(2,2^s)$ whose block systems come from the projective points, are precisely the Maru\v si\v c-Scapellato digraphs with the correct correspondence chosen between the blocks of $\mathcal D_\ell$ and the elements of $\PG(1,2^s)\times \Z_q$. They describe explicitly how certain matrices act on elements of $\PG(1,2^s)\times \Z_q$.

The action on the first coordinate is straightforward; the set of blocks of $\mathcal D_\ell$ that lie in a particular projective point will correspond to the set of vertices of the digraph whose label has that first coordinate. Thus, any matrix will map a vertex whose first coordinate is some projective point, to a vertex whose first coordinate is the image of that projective point under that matrix. However, the action on the second coordinate is less clear, and this is what they explain in more detail.

In Equations~(10) and~(12) of \cite{MarusicS1993}, they explain that the labeling of the vertices is chosen so that $k_\omega(\infty,r)=(\infty,r+1)$ (where $r \in \Z_q$), and $k_\omega(c,r)=(c\omega,r+1)$ for any projective point $c$ other than $\infty$. They also introduce a family of matrices $$h_b=\left[\begin{matrix}1&b\\0&1\end{matrix}\right],$$ where $b \in \F_{2^s}$. Observe that $H=\{h_b: b \in \F_{2^s}\}$ is a group, and in fact since $\F_{2^s}$ has characteristic $2$, $H$ is an elementary abelian $2$-group.
They note in the paper that the stabilizer of $\infty$ in $\SL(2,2^s)$ is the set of upper triangular matrices, and this is generated by $k_\omega$ together with $H$. In Equation~(14), they observe that under their labeling, $h_b(\infty,r)=(\infty,r)$, and $h_b(c,r)=(c+b,r)$ when $c$ is any projective point other than $\infty$.

From this point on, our assumptions that $s=2^t$, $p=2^s+1$ is a Fermat prime, and $q$ divides $p-2$ become important.
With the information we now have in hand, we are ready to understand how diagonal matrices act on the vertex labels from $\PG(1,2^s)\times \Z_q$; this will be valuable to us.

\begin{lem}\label{scalar in MS-labeling}
Let $\omega$ be a primitive root of $\F_{2^s}$, so that $\sqrt{\omega}$ is also a primitive root of $\F_{2^s}$.  Then the permutation $\sqrt{\omega}I_2$ acts on vertices labeled with elements of $\PG(1,2^s) \times \Z_q$ by satisfying

$$\sqrt{\omega} I_2(\infty,r) = (\infty, r + 1)\qquad{\text and}\qquad\sqrt{\omega}I_2(c,r) = (c,r - 1),\text{ when } c\neq \infty.$$
\end{lem}

\begin{proof}
Observe that a point of $\F_{2^s}^2$ that lies in the projective point $\infty$ has $0$ as its second entry, so the action of $k_\omega$ on such a point must be identical to the action of $\sqrt{\omega}I_2$. Thus,  any set $D \in \mathcal D_\ell$ of points lying in the projective point $\infty$ must have the same image under $\sqrt{\omega}I_2$ as under $k_{\omega}$. If $D$ corresponds to the vertex labelled $(\infty,r)$, then since Maru\v si\v c and Scapellato have told us that $k_\omega(\infty,r)=(\infty,r+1)$, it must also be the case that $\sqrt{\omega}I_2(\infty,r)=(\infty,r+1)$.

Similarly, a point of $\F_{2^s}^2$ that lies in the projective point $0$ has $0$ as its first entry, so the action of $k_{\omega^{-1}}=(k_\omega)^{-1}$ on such a point must be identical to the action of $\sqrt{\omega}I_2$. Again,
Maru\v si\v c and Scapellato have told us that $k_\omega(0,r-1)=(0,r)$, it must be the case that $\sqrt{\omega}I_2(0,r)=(0,r-1)$.

Finally, consider any point of $\F_{2^s}^2$ that lies in the projective point $c$ where $c \neq 0, \infty$, so $c=\sqrt{\omega}^t$ for some $t$. Then the point of $\F_{2^s}^2$ has the form $(\sqrt{\omega}^{qi+t+m},\sqrt{\omega}^{qi+m})$ for some $0 \le i \le \ell-1$ and $0 \le m \le q-1$. Straightforward calculations using the field's characteristic of $2$ show that the action of the matrix $\sqrt{\omega}^{-1}I_2$ has the same effect on such a point as the action of the matrix
$$h_ck_\omega h_c=\left[\begin{matrix}\sqrt{\omega} & \sqrt{\omega}^t(\sqrt{\omega}+\sqrt{\omega}^{-1})\\ 0 & \sqrt{\omega}^{-1}\end{matrix}\right].$$
Using the information from Maru\v si\v c and Scapellato, we know that $h_c(c,r)=(c+c,r)=(0,r)$, $k_\omega(0,r)=(0,r+1)$, and $h_c(0,r+1)=(c,r+1)$. Thus, we must also have $\sqrt{\omega}^{-1}I_2(c,r)=(c,r+1)$, and hence $\sqrt{\omega}I_2(c,r)=(c,r-1)$.
\end{proof}

Let $F:\F_{2^s}\mapsto\F_{2^s}$ be the Frobenius automorphism, and so be given by $F(x) = x^2$.  The Frobenius automorphism induces an automorphism $f$ of $\GL(2,2^s)$ in the natural way - by applying $F$ to the entries of the standard matrix of an element of $\GL(2,2^s)$.
Observe that since the Frobenius automorphism is an automorphism, we have $Z \cap \SL(2, 2^s)=\{I_2\}$. Furthermore, every element of $\mathbb F_{2^s}$ is a square, and so every element of $\mathbb F_{2^s}$ arises as the determinant of some matrix in $Z$. Therefore $\langle \SL(2,2^s),Z\rangle=\GL(2,2^s)$. Since we know that $\SL(2,2^s),Z \triangleleft \GL(2,2^s)$, this implies that $\GL(2,2^s)=\SL(2,2^s)\times Z$.

We need to introduce some additional notation that will be used throughout the remainder of this section. We use $\GammaL(2,2^s)$ to denote the group $\GL(2,2^s)\rtimes\langle f\rangle$. We also use $\SigmaL(2,2^s)$ to denote $\SL(2,2^s)\rtimes\langle f\rangle$. We know that $\GL(2,2^s)=\SL(2,2^s)\times Z$, so $\GammaL(2,2^s)=(\SL(2,2^s)\times Z)\rtimes\langle f\rangle$ where the action of $f$ leaves $\SL(2,2^s)$ and $Z$ invariant.

\begin{lem}
\label{diagonals don't normalize2}
Let $p = 2^s + 1$ be a Fermat prime and $q\vert(2^s - 1)$ a prime, with $q\ell=(2^s-1)$. Let $a$ be the order of $2$ modulo $q$, let $b$ be a divisor  of $\gcd(a,s)$ with $b \neq a$, and let $1\neq L = \la f^b \ra$ (where $f$ is the automorphism of $\GL(2,2^s)$ induced by the Frobenius automorphism, as described above).

If $1\not = z/{\cal D}_\ell\in Z/{\cal D}_\ell$ , then $z^{-1} \la \SL(2,2^s),L \ra z/{\cal D}_\ell \not = \la \SL(2,2^s),L \ra/{\cal D}_{ \ell}$.
\end{lem}

\begin{proof}
Let $1\neq L =\la f^b \ra\leq \la f \ra$ and let $G= \la \SL(2,2^s),L \ra$.
Towards a contradiction, suppose that $1 \not = z/{\cal D}_\ell\in Z/{\cal D}_\ell$, and $z^{-1} G z/{\cal D}_\ell = G/{\cal D}_{ \ell}$.  Let $Y = \la z/{\cal D}_\ell\ra$.  As $Z\tl \GammaL(2,2^s)$ is cyclic and $Y$ is the unique subgroup of $Z/{\cal D}_\ell$ of order $\vert Y\vert$, $Y\tl \GammaL(2,2^s)/{\cal D}_\ell$.
Since $z^{-1} G z/{\cal D}_\ell = G/{\cal D}_{ \ell}$, it follows that   $G/{\cal D}_\ell\tl \la Y,G/{\cal D}_\ell\ra$. Moreover, since $Y\cap G/{\cal D}_\ell = 1$ (this follows from $\GL(2,2^s)=\SL(2,2^s) \times Z$), we see $\la Y,G/{\cal D}_\ell\ra\cong Y\times G/{\cal D}_\ell$.
In particular, $z/{\cal D}_\ell$ commutes with $f^b/{\cal D}_\ell$.

Choose $i$ such that $z=\sqrt{\omega} ^i I_2\in Z$, for some fixed generator $\omega$ of $\mathbb F_{2^s}^*$ ($\sqrt{\omega}$ also generates $\mathbb F_{2^s}^*$). It is straightforward to verify that $z^{-1}f^bzf^{-b} =z^{2^b-1}=\sqrt{\omega}^{i(2^b-1)} I_2$.
On the other hand, since $z/{\cal D}_\ell$ commutes with $f/{\cal D}_\ell$,
it follows that $(z^{-1}f^bzf^{-b})/{\cal D}_\ell=1$, implying $\sqrt{\omega}^{i(2^b-1)} I_2/{\cal D}_\ell=1$.

Observe that each block of $\mathcal D_\ell$ has the form $\{(x\sqrt{\omega}^{qj},y\sqrt{\omega}^{qj}): 0 \le j <\ell\},$ for some $x,y \in \mathbb F_{2^s}$. This implies that $\fix_{Z}({\cal D}_\ell)=\la \sqrt{\omega}^qI_2 \ra $. Therefore $\sqrt{\omega}^{i(2^b-1)} I_2/{\cal D}_\ell=1$ if and only if $\sqrt{\omega}^{i(2^b-1)}\in \la \sqrt{\omega}^q \ra$.
We conclude that $i(2^b-1)\equiv q \pmod{2^s-1}$. Since $q$ divides $2^s-1$, it follows that $q$ divides $i(2^b-1)$. Recall that $a$ is the order of $2$ modulo $q$ and $b<a$. This implies that $2^b\not \equiv 1 \pmod{q}$ and therefore $q$ does not divide $2^b-1$.
Since $q$ is a prime, and $q$ divides $i(2^b-1)$, this implies that $q$ divides $i$. However, this means that $z\in \la \sqrt{\omega}^qI_2 \ra= \fix_{Z}({\cal D}_\ell)$ contradicting the assumption  that $ z/{\cal D}_\ell \not =1$.
This contradiction establishes that $z^{-1}Gz/{\cal D}_\ell\not = G/{\cal D}_\ell$, as claimed.
\end{proof}

The first error that we will correct concerns the classification of symmetric Maru\v si\v c-Scapellato graphs given in \cite[Theorem, as it relates to (3.8)]{PraegerWX1993}.  In that paper, Lemma 4.9(a) states that the $q$ connected orbital graphs $X(2^s,q,\emptyset,\{t\})$ (where $t \in \Z_{q}$) of $\SL(2,2^s)$ all have automorphism group $\SigmaL(2,2^s)$ (this group is written in \cite{PraegerWX1993} as $\GammaL(2,2^s)$, but it is clear from the proof of \cite[Theorem 3.7]{PraegerWX1993} that they mean the group we are denoting by $\SigmaL(2,2^s)$).
Using Lemma~\ref{scalar in MS-labeling} to understand the action of $Z$ on these graphs, we see that for any $z \in Z$ with $z \neq 1$, there exists some $t' \in \Z_q^*$ such that $X(2^s,q,\emptyset,\{t\})^z=X(2^s,q,\emptyset,\{t-t'\})$ (more precisely, if $z=\sqrt{\omega}^iI_2$, then $t'=2i$). Thus, every such $z$ acts as a cyclic permutation on this set of $q$ graphs.
Suppose that $\Gamma$ and $\Gamma^z$ are two of these orbital digraphs with $z/\mathcal D_\ell \neq 1$ (so that the graphs are distinct), and $\Aut(\Gamma)=\SigmaL(2,2^s)$ as claimed in \cite{PraegerWX1993}.
Then $\Aut(\Gamma^z)=z^{-1}(\Aut(\Gamma))z,$ and by taking $b=1$ in Lemma~\ref{diagonals don't normalize2} we see that
$\SigmaL(2,2^s)/\mathcal D_\ell \neq z^{-1}\SigmaL(2,2^s)z/\mathcal D_\ell$,
contradicting their claim that $\Aut(\Gamma^z)=\Aut(\Gamma)$.

The mathematical error leading to the incorrect statement of \cite[Lemma 4.9]{PraegerWX1993} actually arises in \cite[Lemma 4.8]{PraegerWX1993} where it is concluded that the automorphism group $G$ of any Maru\v si\v c-Scapellato graph satisfies $\SL(2,2^s)/{\cal D}_\ell\le G\le \SigmaL(2,2^s)/{\cal D}_\ell$ (using our notation).   The proof of \cite[Lemma 4.8]{PraegerWX1993} only gives that $G/{\cal B} = \SigmaL(2,2^s)/\PG(1,2^s) = \PSigmaL(2,2^s)$.  If we consider any of the groups that are conjugate to $\SigmaL(2,2^s)$ by a scalar matrix, which we have shown in Lemma \ref{diagonals don't normalize2} are distinct modulo $\mathcal D_\ell$, the fact that scalar matrixes fix every point of $\PG(1,2^s)$ shows that every such group satisfies this equation.  With that said, the proof of \cite[Lemma 4.9 (b)]{PraegerWX1993} is correct if we strengthen the hypothesis to assume that $\SL(2,2^s)/{\cal D}_\ell\le G\le\SigmaL(2,2^s)/{\cal D}_\ell$. So we can restate their result correctly as follows, to identify the symmetric Maru\v si\v c-Scapellato digraphs whose automorphism group is contained in $\SigmaL(2,2^s)/\mathcal D_\ell$.

Note that when $G=\Aut(\Gamma)$ where $\Gamma$ is one of these Maru\v si\v c-Scapellato graphs, and $\SL(2,2^s)/{\cal D}_\ell\le G\le\SigmaL(2,2^s)/{\cal D}_\ell$, all of these actions on $\mathcal D_\ell$ are faithful, so that $\SL(2,2^s)/{\cal D}_\ell\cong \SL(2,2^s)$, $G/\mathcal D_\ell \cong G$, and $\SigmaL(2,2^s)/\mathcal D_\ell \cong \SigmaL(2,2^s)$. In \cite{PraegerWX1993}, they were to some extent studying the abstract structure of these groups, and did not make this distinction, which may have contributed to the confusion and does lead to our statement looking somewhat different from theirs.

\begin{thrm}[see \cite{PraegerWX1993}, Lemma 4.9]\label{PraegerWXsymmetric}
Let $p = 2^s + 1$ be a Fermat prime and $q\vert(2^s - 1)$ be prime. Let $\Gamma = X(2^s,q,S,T)$ be a symmetric Maru\v si\v c-Scapellato digraph and assume that $\SL(2,2^s)/\mathcal D_\ell\le \Aut(\Gamma) \le \SigmaL(2,2^s)/\mathcal D_\ell$.  Let $a$ be the order of $2$ modulo $q$.  Then $S = \emptyset$ and one of the following is true:
\begin{enumerate}
\item $T = \{0\}$, $\Gamma$ has valency $q$, and automorphism group $\SigmaL(2,2^s)/\mathcal D_\ell$.
\item There is a divisor $b$ of $\gcd(a,s)$ and $1 < a/b < q - 1$ such that $T = U_{b,i} = \{i2^{bj}:0\le j < a/b\}$.  There are exactly $(q-1)/a$ distinct graphs of this type for a given $b$, each of valency $qa/b$, and the automorphism group of each is $\la \SL(2,2^s),L\ra /{\cal D}_\ell$ where $L\le \la f\ra$ is of order $s/b$.  Up to isomorphism, there are exactly $(q-1)/b$ such graphs.
\end{enumerate}
\end{thrm}

Before turning to the characterization of symmetric Maru\v si\v c-Scapellato graphs of order $qp$, we will need a solution to the isomorphism problem for these graphs.  This problem has been solved in \cite{Dobson2016}, but the solution there is not suited to our needs.  The solution given in \cite{Dobson2016} is also perhaps not optimal in the sense that it requires one check $\vert\SigmaL(2,2^s)\vert = \vert\Aut(\SL(2,2^s))\vert$ maps to determine isomorphism, while we show in the next result that one only needs to check $qs$ maps.

\begin{thrm}\label{New MS iso}
Let $p = 2^s + 1$ be a Fermat prime, $q\vert(2^s - 1)$ a prime, and $\Gamma,\Gamma'$ be non-degenerate Maru\v si\v c-Scapellato digraphs.  Then $\Gamma$ and $\Gamma'$ are isomorphic if and only if $\delta(\Gamma) = \Gamma'$, where $\delta\in\la Z,f\ra/{\cal D}_\ell$.
\end{thrm}

\begin{proof}
It is shown in \cite[Theorem 1]{Dobson2016} that $\Gamma' = \delta(\Gamma)$ for some $\delta$ if and only if this occurs for a $\delta$ that normalizes $\SL(2,2^s)$.  This normalizer is $\GammaL(2,2^s)$ as every element of $S_{qp}$ that normalizes $\SL(2,2^s)$ can be written in the form $abc$, where $c\in\SL(2,2^s)$, $b\in\Aut(\SL(2,2^s))$, and $a$ is contained in the centralizer of $\SL(2,2^s)$ in $S_{qp}$.  As $\SL(2,2^s) = \PSL(2,2^s)$ and $\Aut(\PSL(2,2^s)) = \PSigmaL(2,2^s)$, we may take $b\in\la f\ra$.  As the centralizer in $S_{qp}$ of $\SL(2,2^s)$ has order $q$ by \cite[Theorem 4.2A (i)]{DixonM1996} and \cite[Lemma 2.1]{MarusicS1993}, we see that $a\in Z/{\cal D}_\ell$.
\end{proof}



We are now ready to determine the symmetric Maru\v si\v c-Scapellato digraphs of order a product of two distinct primes with imprimitive automorphism group.

\begin{thrm}\label{MS symmetric}
Let $s = 2^t$, $p = 2^s + 1$ be a Fermat prime, and $q\vert(2^s - 1)$ be prime. Let $\Gamma = X(2^s,q,S,T)$ be a nondegenerate symmetric Maru\v si\v c-Scapellato digraph constructed with the primitive root $w$ of $\F_{2^s}$ with an imprimitive automorphism group.  Let $a$ be the order of $2$ modulo $q$, and $d = \sqrt{w}I$.  Then $S = \emptyset$ and one of the following is true:
\begin{enumerate}
\item $T = \{-2k\}$, $\Gamma$ has valency $q$, and automorphism group $d^{-k}\SigmaL(2,2^s)d^k/{\cal D}_\ell$, $k\in\Z_q$.
\item There is a divisor $b$ of $\gcd(a,2^s)$, $1 < d/e < q - 1$, and $k\in\Z_q$ such that $T = U_{b,i,k} = \{i2^{bj} - 2k:0\le j < a/b\}$.  There are exactly $(q-1)b/a$ distinct graphs of this type for a given $b$ and $k$, each of valency $qa/b$, and the automorphism group of each is $d^{-k}\la SL(2,2^s),L\ra d^k/{\cal D}_{\ell}$ where $L\le \la f\ra$ is of order $2^s/b$.  Up to isomorphism, there are exactly $(q-1)/b$ such graphs.
\end{enumerate}
\end{thrm}

\begin{proof}
For the proof of this result, we will abuse notation by writing $H$ instead of $H/{\cal D}_\ell$ where $H/{\cal D}_\ell\le d^{-k}\SigmaL(2,2^s)d^k/{\cal D}_\ell$, and will similarly abuse notation for elements of $d^{-k}\SigmaL(2,2^s)d^k/{\cal D}_\ell$.  This should cause no confusion.  The result follows by Theorem \ref{PraegerWXsymmetric} if $\Aut(\Gamma)\le \SigmaL(2,2^s)/$, in which case $k = 0$.  Suppose that $\Aut(\Gamma)$ is not contained in $\SigmaL(2,2^s)/$.  As $\Aut(\Gamma)$ is imprimitive and contains $\SL(2,2^s)/$, by \cite[Theorem]{MarusicS1992} either $\Gamma$ is metacirculant or the only invariant partition of $\Aut(\Gamma)$ is ${\cal B} = \PG(1,2^s)$ which is also the only invariant partition of $\SL(2,2^s)$.  If $\Gamma$ is metacirculant, then it is degenerate by \cite[Theorem 2.1]{MarusicS1994}.  Hence $\Gamma$ is not metacirculant and so $\fix_{\Aut(\Gamma)}({\cal B}) = 1$ by \cite[Theorem 3.4]{Marusic1988}.  Then $\Aut(\Gamma)/{\cal B}\cong\Aut(\Gamma)$ is a group of prime degree $p$.  By \cite[Corollary 3.5B]{DixonM1996} we have $\Aut(\Gamma)/{\cal B}\le\AGL(1,p)$ or is a doubly-transitive group.  By \cite[Theorem 4.1B]{DixonM1996} we see either $\Aut(\Gamma)/{\cal B}\le\AGL(1,p)$ or is a doubly-transitive group with nonabelian simple socle.  If $\Aut(\Gamma)/{\cal B}\le\AGL(1,p)$ then $\Aut(\Gamma)$ contains a normal subgroup of order $p$, and so has blocks of size $p$, a contradiction.  Thus $\Aut(\Gamma)/{\cal B}$ is a doubly-transitive group with nonabelian simple socle.  By \cite[Lemmas 4.5, 4.6, and 4.7]{PraegerWX1993} we have $\SL(2,2^s)\tl\Aut(\Gamma)$.

Now, $N_{S_V}(\SL(2,2^s)) = \la \SL(2,2^s),f,Z\ra=\GammaL(2,2^s)= (\SL(2,2^s)\times Z)\rtimes \la f\ra$.   Thus every element of $N_{S_V}(\SL(2,2^s))$, and hence $\gamma\in\Aut(\Gamma)$, can be written as $\gamma = f^iz\omega$, where $\omega\in \SL(2,2^s)$, $z\in Z$, and $i$ is a positive integer.  Of course, as $\SL(2,2^s)\le\Aut(\Gamma)$, $f^iz\in\Aut(\Gamma)$ if and only if $f^iz\omega\in\Aut(\Gamma)$ for some $\omega\in\SL(2,2^s)$.  Then $H = \Aut(\Gamma)\cap\{f^iz:i\in\Z,z\in Z\}$ is a subgroup of $\Aut(\Gamma)$, and $\Aut(\Gamma)/\SL(2,2^s)\cong H$ by the First Isomorphism Theorem.  As $\Aut(\Gamma)\cap Z = 1$, $H$ is isomorphic to a subgroup of $\la f\ra$, and hence   $\Aut(\Gamma)/\SL(2,2^s)$ is isomorphic to a cyclic $2$-subgroup.  Then $\Aut(\Gamma)/\SL(2,2^s)$ is conjugate by an element $z/\SL(2,2^s)\in Z/\SL(2,2^s)$ to a subgroup of $\la f\ra/\SL(2,2^s)$, and so $z^{-1}\Aut(\Gamma)z\le \SigmaL(2,2^s)$.  Then $\Aut(\Gamma)\le z\SigmaL(2,2^s)z^{-1}$ and $z^{-1}(\Gamma)$ is a nondegenerate symmetric Maru\v si\v c-Scapellato digraph with $\Aut(z^{-1}(\Gamma))\le\SigmaL(2,2^s)$, and so is given by Theorem \ref{PraegerWXsymmetric}.

In order to verify the numbers of symmetric Maru\v si\v c-Scapellato digraphs are as in the result, we need only to see different scalar matrices do indeed give different symmetric Maru\v si\v c-Scapellato digraphs. Suppose that there exist two non-isomorphic symmetric Maru\v si\v c-Scapellato graphs $\Gamma_1$ and $\Gamma_2$ with automorphism groups contained in $\SigmaL(2,2^s)$, such that $z_1(\Gamma_1)=z_2(\Gamma_2)$, for $z_1,z_2\in Z$.
Then $\Gamma_2=z_2^{-1}z_1(\Gamma_1)$. This implies that $\Gamma_1$ and $\Gamma_2$ are of the same valency and since by Theorem~\ref{PraegerWXsymmetric} all Maru\v si\v c-Scapellato graphs with the same valency have the same automorphism groups, it follows that $\Aut(\Gamma_1)=\Aut(\Gamma_2)=\la SL(2,2^s), L \ra$, where $L=\la f^b \ra$. By Lemma~\ref{diagonals don't normalize2} it follows that $z^{-1}\la SL(2,2^s), L \ra z= \la SL(2,2^s), L \ra$ holds only when $z=1$.
On the other hand, since $\Gamma_2=z_2^{-1}z_1(\Gamma_1)$ it follows that $\Aut(\Gamma_2)=(z_2^{-1}z_1) \Aut(\Gamma_1)(z_2^{-1}z_1)^{-1}$, and hence $z_1=z_2$, which implies that different scalar matrices do indeed give different symmetric Maru\v si\v c-Scapellato graphs.

Let $z^{-1}(\Gamma) = X(2^s,q,\emptyset,  T)$, where $T = \{0\}$ or $T = U_{b,i} = \{i2^{bj}:0\le j < a/b\}$.  Let $z = d^k$ for some positive integer $k$. We need only verify that $T = \{2k\}$ or $U_{b,i,k} = \{i2^{bj} + 2k:0\le j < a/b\}$.  Now, in $z^{-1}(\Gamma)$, the neighbors of $(\infty,r)$ are $\{(y,r + u):y\in\F_{2^s},u\in U\}$.  Considering $z(z^{-1}(\Gamma)) = \Gamma$ and applying Lemma \ref{scalar in MS-labeling}, we see the neighbors of $(\infty,r + k)$ in $\Gamma$ are $\{(y,r + u - k):y\in\F_{2^s},u\in U\}$.  Equivalently, the neighbors of $(\infty,r)$ in $\Gamma$ are $\{(y,r +u - 2k):y\in\F_{2^s},u\in U\}$ and the result follows.
\end{proof}

We now determine the full automorphism group of any Maru\v si\v c-Scapellato digraph.

\begin{thrm}
Let $p = 2^s + 1$ be a Fermat prime, $q\vert(2^s - 1)$ be prime, and $\Gamma$ be a Maru\v si\v c-Scapellato digraph of order $qp$.  Then $\Gamma$ or its complement is $X(2^s,q,S,T)$ and one of the following is true.
\begin{enumerate}
\item $\Aut(\Gamma)$ is primitive and
\begin{enumerate}
\item $s=2$, $qp = 15$, $S = \Z_3^*$ and $T = \{0\},\{1\}$, or $\{2\}$.  Then $\Gamma$ is isomorphic to the line graph of $K_6$ and has automorphism group $d^{-1}\SigmaL(2,4)d\cong S_6$ for some $d\in Z$.
\item $p = k^2 + 1$, $q = k + 1$, $S = \Z_q^*$ and $\vert T\vert = 1$.  Then there exists $d\in Z/{\cal D}_\ell$ such that $\Aut(\Gamma) =  d^{-1}{\rm P}\Gamma{\rm Sp}(4,k)d$.
\item $S = \Z_q^*$, $T = \Z_q$, and $\Gamma$ is a complete graph with automorphism group $S_{qp}$.
\end{enumerate}
\item $\Aut(\Gamma)$ is imprimitive and
\begin{enumerate}
\item $S < \Z_q^*$, $T = \Z_q$, $\Gamma$ is degenerate, and $\Aut(\Gamma)\cong S_p\wr \Aut(\Cay(\Z_q,S))$.
\item In all other cases there exists $L\le \la f/{\cal D}_\ell\ra$ and $d\in Z/{\cal D}_\ell$ such that
$$\Aut(\Gamma) = d^{-1}\la\SL(2,2^s), L\ra d/{\cal D}_\ell$$
\noindent which is isomorphic to a subgroup of $\SigmaL(2,2^s)/{\cal D}_\ell$ that contains $\SL(2,2^s)/{\cal D}_\ell$.
\end{enumerate}
\end{enumerate}
\end{thrm}

\begin{proof}
As in the previous result, we will abuse notation and drop the ${\cal D}_\ell$'s from our notation.  The case when $\Aut(\Gamma) = S_{qp}$ is trivial.  The other Maru\v si\v c-Scapellato graphs of order $qp$ with primitive automorphism group were calculated in \cite{MarusicS1994} and their automorphism groups computed in \cite{PraegerX1993}.  This gives the information in the result with $d = 1$.  We observe that $\SigmaL(2,2^s)$ is contained in $\Aut(\Gamma)$, and so the only possible isomorphisms with other Maru\v si\v c-Scapellato graphs are with elements of $Z$ by Lemma \ref{New MS iso}.  That the elements of $Z$ give different graphs follows as they normalize $\SL(2,2^s)$ but are not contained in $\Aut(\Gamma)$.

If $\Aut(\Gamma)$ is imprimitive and $\Gamma$ is degenerate, then $T = \Z_q$ and as $\Aut(\Gamma)$ is imprimitive, $S\not = \Z_q^*$ as otherwise $\Gamma = K_{qp}$ has a primitive automorphism group.  It is then not difficult to see that $\Gamma\cong K_p\wr\Cay(\Z_q,S)$ and by \cite[Theorem 5.7]{DobsonM2009} $\Aut(\Gamma)\cong S_p\wr\Aut(\Cay(\Z_q,S))$.

If $\Aut(\Gamma)$ is imprimitive and $\Gamma$ is non-degenerate, then $\Gamma$ is a generalized orbital digraph of $\Aut(\Gamma)$. We write $\Gamma = \Gamma_1\cup\dotsm \cup \Gamma_r$ where each $\Gamma_i$ is an orbital digraph of $\Aut(\Gamma)$.  Note that as $\Aut(\Gamma)$ is imprimitive, some orbital digraph of $\Aut(\Gamma)$ is disconnected.  Also, each connected orbital digraph of $\Aut(\Gamma)$ is either symmetric or $1/2$-transitive, and as each orbital digraph of $\Gamma$ is a generalized orbital digraph of $\SL(2,2^s)$, we see each connected orbital digraph of $\Aut(\Gamma)$ is symmetric as each connected orbital digraph of $\SL(2,2^s)$ is symmetric.

If there exist connected orbital digraphs of $\Aut(\Gamma)$ that are subdigraphs of $\Gamma$ whose automorphism groups are contained in $d^{-1}\SigmaL(2,2^s)d$ and $e^{-1}\SigmaL(2,2^s)e$ for $d\not = e$ both in $Z$, then
$$\SL(2,2^s)\le\Aut(\Gamma)\le d^{-1}\SigmaL(2,2^s)d\cap e^{-1}\SigmaL(2,2^s)e\le d^{-1}\SigmaL(2,2^s)d.$$
Thus $\Aut(\Gamma)$ is a subgroup of $d^{-1}\SigmaL(2,2^s)d$ that contains $\SL(2,2^s)$.  Now let $\SL(2,2^s)\le K\le\SigmaL(2,2^s)$.  As $\SL(2,2^s)\tl \SigmaL(2,2^s)$, every element of $\SigmaL(2,2^s)$, and consequently every element of $K$, can be written as $gf^c$ for some $g\in\SL(2,2^s)$ and integer $c$.  As $\SL(2,2^s)\le K$, we have $gf^c\in K$ if and only if $f^c$ in $K$.  We conclude $K = \la\SL(2,2^s),L\ra$, where $L\le \la f\ra$ consists of all powers of $f$ contained in $K$.  Then $d^{-1}\SL(2,2^s)d\le\Aut(\Gamma)\le d^{-1}\SigmaL(2,2^s)d$ and $\Aut(\Gamma) = d^{-1}\la \SL(2,2^s),L\ra d$ for some $L\le \la f\ra$ as required.  We thus assume that every connected orbital digraph of $\Aut(\Gamma)$ that is a subdigraph of $\Gamma$ has automorphism group contained in $d^{-1}\SigmaL(2,2^s)d$ for some scalar matrix $d\in Z$.  Suppose that $\Gamma_1,\ldots,\Gamma_t$, $t\le r$ are the connected orbital digraphs of $\Aut(\Gamma)$ that are subdigraphs of $\Gamma$.  Then by Theorem \ref{MS symmetric}, $\Gamma_i$ has automorphism group $d^{-1}\la \SL(2,2^s),L_i\ra d$ where $L_i\le \la f\ra$.  Then $\Aut(\Gamma)\le d^{-1}\la \SL(2,2^s),L'\ra d$ where $L' = \cap_{i=1}^rL_i$.  Finally, let $L$ be the subgroup of $L'$ consisting of automorphisms of the subdigraph of $\Gamma$ obtained by removing all edges between elements of $\PG(1,2^s)/{\cal D}_\ell$.  Then $\Aut(\Gamma) = d^{-1}\la\SL(2,2^s),L\ra d$ and the result follows.
\end{proof}

We remark that the automorphism group of a Maru\v si\v c-Scapellato digraph $\Gamma$ can be calculated quite quickly:

If $\Gamma$ is nondegenerate and $\Aut(\Gamma)$ imprimitive, then one only needs to determine the subgroup of $d^{-1}\SigmaL(2,2^s)d/{\cal D}_\ell = \la \SL(2,2^s),d^{-1}fd\ra/{\cal D}_\ell$ which is $\Aut(\Gamma)$.  In particular, one only needs to determine the maximal subgroup of $d^{-1}\la f\ra d$ contained in $\Aut(\Gamma)$.  This can easily be accomplished as all such subgroups can be computed quickly. Let $s = 2^t$, $t\ge 1$.  As $F(x) = x^2$, $f$ has order $2^t$, so there are $t+1$ subgroups of $\la f\ra$ each determined by a generator of the form $f^{2^r}$, $0\le r\le t$.  As $Z/{\cal D}_\ell$ has order at most $2^s$, there are at most $(t+1)2^s$ maps which need to be tested as elements of $\Aut(\Gamma)$ in order to determine $\Aut(\Gamma)$.

If $\Gamma$ is degenerate or $\Aut(\Gamma)$ is primitive, then this can be determined easily as the sets $S$ and $T$ are given explicitly.  Again, one only needs to determine $d$, and this can be done as above by checking which $d^{-1}gd$ is contained in $\Aut(\Gamma)$.

\section{Missing digraphs whose automorphism group is primitive}\label{sec:prim-errors}

The first error in the literature is most probably simply an unfortunate typographical error.  The misprint occurs in \cite[Table 3]{LiebeckS1985a} for the groups $\PSL(2,q)$ of degree $q(q^2-1)/24$ with point stabilizer $A_4$.  In the ``Comment" column, the paper literally lists ``$q\equiv +3\ (\mod 8), q\le 19$".  Of course, as written the ``$+$" is entirely superfluous, but in reality it should be a ``$\pm$".  Indeed, without the $\pm$ the group $\PSL(2,13)$ which has $A_4$ as a maximal subgroup is not listed.  The action of $\PSL(2,13)$ on right cosets of $A_4$ is primitive of degree $\vert\PSL(2,13)\vert/\vert A_4\vert = 7\cdot 13$. The authors thank Primo\v z Poto\v cnik for pointing out this error.

\begin{lem}\label{PSL213lem}
Let $\PSL(2,13)$ act transitively on $7\cdot 13$ points with point-stabilizer $A_4$.

Then there are $3$ self-paired orbitals of size $4$ all of which are $2$-arc-transitive.  No other orbital digraphs are $2$-arc-transitive.
Two of the graphs corresponding to these self-paired orbitals are isomorphic with automorphism group $\PSL(2,13)$. The graph corresponding to the union of these orbitals is symmetric and has automorphism group $\PGL(2,13)$.
The  graph corresponding to the remaining orbital has automorphism group $\PGL(2,13)$ and is symmetric.

There is $1$ self-paired orbital of size $6$ whose corresponding graph is symmetric and has automorphism group $\PGL(2,13)$.

There are $2$ non self-paired orbitals of size $12$ whose corresponding digraphs have automorphism group $\PSL(2,13)$, and whose union corresponds to a symmetric graph with automorphism group $\PGL(2,13)$.

There are $4$ self-paired orbitals of size $12$ that are all symmetric, two of which correspond to graphs that are isomorphic with automorphism group $\PSL(2,13)$. Their  union corresponds to a graph that has automorphism group $\PGL(2,13)$ and is symmetric.
The remaining two self-paired orbitals correspond to graphs that are non-isomorphic and have automorphism group $\PGL(2,13)$ and are symmetric.

Any other digraph of order $91$ that contains $\PSL(2,13)$ as a transitive subgroup and is not complete or the complement of a complete graph is a union of the above digraphs and is not symmetric. It will have automorphism group either $\PSL(2,13)$ or $\PGL(2,13)$, depending upon whether or not it can be written as a union of graphs all of whose automorphism groups are $\PGL(2,13)$. If this is possible, then it has automorphism group $\PGL(2,13)$; otherwise, its automorphism group will be $\PSL(2,13)$.
\end{lem}

\begin{proof}
The information about the orbital digraphs of $\PSL(2,13)$, including whether or not they are self-paired and their automorphism groups and whether they are $2$-arc-transitive or symmetric, was obtained using MAGMA.  So was information about the automorphism groups of unions of exactly two orbital digraphs.  It thus remains to determine the automorphism group of any other digraph of order $91$ that contains $\PSL(2,13)$ and is not complete or its complement.

Let $\Gamma$ be such a digraph.  Then $\Aut(\Gamma)\not = S_{91}$, and $\Aut(\Gamma)$ is $2$-closed.  There is only one other socle of a primitive but not $2$-transitive subgroup of $S_{91}$, namely $\PSL(3,9)$ by \cite[Table B.2]{DixonM1996}.  However, $\PSL(3,9)$ contains no subgroup isomorphic to $\PSL(2,13)$ by \cite{Bloom1967}.  Thus $\soc(\Aut(\Gamma)) = \PSL(2,13)$ and so $\Aut(\Gamma) = \PSL(2,13)$ or $\PGL(2,13)$.  Clearly, if $\Gamma$ can be written as a union of graphs whose automorphism group is $\PGL(2,13)$, then $\Aut(\Gamma) = \PGL(2,13)$.  Otherwise, by the first part of this lemma, $\Gamma$ is a union of digraphs one of which has automorphism group $\PSL(2,13)$ but is not invariant under $\PGL(2,13)$ and its different image under $\PGL(2,13)$ is not a subdigraph of $\Gamma$.  Hence $\Aut(\Gamma) = \PSL(2,13)$.
\end{proof}

This leads to the next error in the literature, which is also mainly typographical.  Namely, in \cite[Table 2]{PraegerX1993} the entries for $\PSL(2,p)$ require $p\ge 11$.  For $p = 5$, $\PSL(2,5)\cong A_5$ is $2$-transitive in its representation of degree $6$, and so any digraph of order $6$ whose automorphism group contains $\PSL(2,5)$ is necessarily complete or has no arcs and has automorphism group $S_6$.  For $\PSL(2,7)\cong\PSL(3,2)$, we see from Theorem \ref{gendigraphauto} that there are other digraphs that are not graphs that are not listed in \cite[Table 2]{PraegerX1993}.  The error here is more one of omission than a mistake in the proof - in \cite{PraegerX1993} the proofs are for vertex-transitive digraphs and graphs of order at least $5p$ (see for example \cite[Table IV]{PraegerX1993}), $p\ge 7$, as the case when $p = 3$ was already considered in \cite{WangX1993} - but \cite{WangX1993} did not consider digraphs that were not graphs.

The next error involves $M_{23}$ in its actions on $11\cdot 23$ points.  There are two actions of $M_{23}$ on $253 = 11\cdot 23$ points. One is on pairs taken from a set of $23$ elements, while the other is on the septads (sets of size $7$) in the Steiner system $S(4,7,23)$ \cite{Conway1985}.  The action on pairs gives $M_{23}$ as a transitive subgroup of $\Aut(T_{23})$, the triangle graph, whose automorphism group is $S_{23}$, and this graph is listed in the row corresponding to $A_{23}$.  The action of $M_{23}$ on septads was not considered in \cite{PraegerX1993}.

\begin{lem}\label{M23lem}
The action of $M_{23}$ on septads (sets of size $7$) in the Steiner system $S(4,7,23)$ of degree $11\cdot 23$ has two orbital digraphs which are graphs of valency $112$ and $140$. Both of these graphs are Cayley graphs of the nonabelian group of order $11\cdot 23$ and so are also isomorphic to metacirculant graphs.  Both graphs have automorphism group $M_{23}$ and neither is $2$-arc-transitive.
\end{lem}

\begin{proof}
The action on septads gives $M_{23}$ as a transitive subgroup of $\Aut(M_{23})$.  By \cite{Atlas} the suborbits are of length $112$ and $140$, and by \cite{Conway1985} there is a maximal subgroup $H$ of $M_{23}$ of order $253$, and $H$ is isomorphic to the Frobenius group of order $253$.  By order arguments no element of $H$ is contained in the stabilizer of a septad, and so $H$ must be semiregular.  By the Orbit-Stabilizer Theorem we see that $H$ is regular.  As Sabidussi showed \cite{Sabidussi1958} that a graph is isomorphic to a Cayley graph of the group $G$ if and only if it contains a regular subgroup isomorphic to $G$, each of the two orbital graphs of $M_{23}$ are Cayley graphs.  As every Cayley graph of order $qp$ is a metacirculant graph, these two graphs are also metacirculant.

Turning to the automorphism groups of the two orbital digraphs $\Gamma_1$ and $\Gamma_2$ of $M_{23}$, they are complements of each other and so $\Aut(\Gamma_1) = \Aut(\Gamma_2)$.  Also, with respect to the $2$-closure of this action of $M_{23}$ (which can be defined as the intersection of the automorphism groups of its orbital digraphs), we have $M_{23}^{(2)} = \Aut(\Gamma_1)\cap\Aut(\Gamma_2) = \Aut(\Gamma_1)$.  By \cite[Theorem 1]{LiebeckPS1988a} we have $M_{23}\tl\Aut(\Gamma_1)$.  By \cite[Table B.2]{DixonM1996} we have $\Aut(\Gamma_1) = M_{23}$.

Finally, in order to be $2$-arc-transitive, $d(d-1)$ must divide the order of the stabilizer of a point in $M_{23}$ where $d$ is the valency of $\Gamma_1$ or $\Gamma_2$, and this stabilizer has order $2^7\cdot 3^2\cdot 5\cdot 7$.  So neither $\Gamma_1$ nor $\Gamma_2$ is $2$-arc-transitive.
\end{proof}

\begin{table}
\begin{center}
\begin{tabular}{| c | c | c | c | c |}
\hline
$\soc(G)$ & $qp$ & Valency & Cayley & Reference\\
\hline
$A_{qp}$ & $qp$ & $0,qp-1$ & Y &  \\
\hline
$A_p$ & $\frac{p(p-1)}{2}$ & $2(p-2),\frac{(p-2)(p-3)}{2}$ & Y* &  \cite[3.1]{PraegerX1993}\\
\hline
$A_{p+1}$ & $\frac{p(p+1)}{2}$ & $2(p-1),\frac{(p-1)(p-2)}{2}$ & ${\rm N}^\dagger$ & \cite[3.1]{PraegerX1993} \\
\hline
$A_7$ & $5\cdot 7$ & ${\bf 4},12,18$ & N & \cite[3.2]{PraegerX1993} \\
\hline
$\PSL(4,2)$ & $5\cdot 7$ & $16,18$ & N & \cite[3.3]{PraegerX1993} \\
\hline
$\PSL(5,2)$ & $5\cdot 31$ & $42,112$ & Y & \cite[3.3]{PraegerX1993} \\
\hline
$\Omega^\pm(2d,2)$ & $(2^d\mp 1)(2^d\pm 1)$ & $2^{2d - 2},2(2^{d-1}\mp 1)(2^{d-2}\pm 1)$ & N & \cite[3.4]{PraegerX1993}\\
\hline
$\PSp(4,k)$ & $(k^2 + 1)(k + 1)$ & $k^2 + k$, $k^3$, $k$ even & ${\rm N}^\dagger$ & \cite[3.5]{PraegerX1993}\\
\hline
$\PSL(2,k^2)$ & $k(k^2 + 1)/2$ & $k^2 - 1,\frac{k^2 - k}{2}, k^2 \pm k$, $k \equiv 1\ (\mod 4)$ & N & \cite[4.1]{PraegerX1993}\\
\hline
$\PSL(2,k^2)$ & $k(k^2 + 1)/2$ & $k^2 - 1,\frac{k^2 + k}{2}, k^2 \pm k$, $k\equiv 3\ (\mod 4)$ & N & \cite[4.1]{PraegerX1993}\\
\hline
$\PSL(2,p)$  & $\frac{p(p \mp 1)}{2}$ & $\frac{p\pm 1}{2}, p\pm 1$, or & Y**  & \cite[4.4]{PraegerX1993}    \\
             &                        & $\frac{p\pm 1}{4}$ or $2(p-1)$ & & \\
\hline
$G = \PGL(2,7)$ & $3\cdot 7$ & $4,8$ & Y & \cite[Example 2.3]{WangX1993} \\
\hline
$G = \PGL(2,11)$ & $5\cdot 11$ & ${\bf 4},6,8,12,24$ & Y & \cite[4.3]{PraegerWX1993}  \\
\hline
$\PSL(2,13)$ & $7\cdot 13$ & ${\bf 4},6,12,24$ & N &  Lemma \ref{PSL213lem} \\
\hline
$\PSL(2,19)$ & $3\cdot 19$ & ${\bf 6},20,30$ & Y & \cite[4.2]{PraegerWX1993} \\
\hline
$\PSL(2,23)$ & $11\cdot 23$ & ${\bf 4},6,8,12,24$  & Y & \cite[4.3]{PraegerWX1993}  \\
\hline
$\PSL(2,29)$ &   $7\cdot 29$   & $12,20,30,60$ & ${\rm N}^\#$ & \cite[4.2]{PraegerWX1993} \\
\hline
$\PSL(2,59)$ & $29\cdot 59$ & ${\bf 6},10,12,20,30,60$ & Y & \cite[4.2]{PraegerWX1993} \\
\hline
$\PSL(2,61)$ & $31\cdot 61$ & ${\bf 6},10,12,20,30,60$ & N & \cite[4.2]{PraegerWX1993} \\
\hline
$M_{22}$    & $7\cdot 11$  & ${\bf 16},60$ & N & \cite[3.6]{PraegerWX1993}             \\
\hline
$M_{23}$    & $11\cdot 23$       &  $112,140$ & Y & Lemma \ref{M23lem} \\
\hline
\end{tabular}
\caption{The graphs of order $pq$ with primitive automorphism groups.}
\label{table}
\end{center}
\end{table}

\begin{thrm}
Let $\Gamma$ be a vertex-transitive graph of order $qp$, where $q$ and $p$ are distinct primes, whose automorphism group $G$ is simply primitive.  Then $\soc(G)$ is given in Table \ref{table}.  There is a boldface entry in the column ``Valency" if and only if there is a $2$-arc-transitive graph of that valency.  The superscipt symbols in the table have the following meanings:
\begin{itemize}
\item $*$ means $p\ge 7$,
\item $\dagger$ means that these graphs are also Maru\v si\v c-Scapellato graphs but in the case of $A_{p+1}$ this is only true for $A_6$,
\item $**$ means these graphs are Cayley if and only if $p\equiv 3\ (\mod 4)$,
\item $\#$ means that these graphs are metacirculant graphs which are not Cayley graphs.
\end{itemize}
\end{thrm}

\begin{proof}
Most of the information in the Table \ref{table} is taken directly from the sources in the column ``Reference", with the following exceptions.  First, information about $2$-arc-transitive graphs not given in Lemma \ref{PSL213lem} or \ref{M23lem} can be found in \cite{MarusicP2002}.  That the generalized orbital digraphs of $\PSL(2,29)$ are metacirculants is proven in \cite[pg. 192, paragraph 3]{MarusicS1994}.  The vertex-transitive graphs of order $pq$ with primitive automorphism group that are also isomorphic to nontrivial Maru\v si\v c-Scapellato graphs are determine in \cite{MarusicS1994} starting at the bottom of page 192.
\end{proof}

\section{Other errors in the literature}\label{sec:other-errors}

To end this paper, we list the errors that we are aware in the literature that follow from the errors above and that are not in the original papers where the error was made.

\begin{itemize}
\item The statement of \cite[Theorem 2.5]{DobsonHKM2017} is missing the graphs given in Theorem \ref{gendigraphauto} with imprimitive automorphism group $\PSL(2,11)$.  This result is only used to discuss graphs of order $21$, and so this error does not affect any results proven in the paper.
\item The result \cite[Proposition 2.5]{WangFZWM2016} does not list the symmetric graphs of valency $4$ given by Lemma \ref{PSL213lem}.  Consequently, \cite[Lemma 3.4]{WangFZWM2016} has a small gap which can be filled using GAP or MAGMA.
\item The result \cite[Proposition 4.2]{MarusicP2002} is missing the $2$-arc-transitive graphs of valency $4$ given by Lemma \ref{PSL213lem}.
\item  The result \cite[Corollary 3.3, Table 1]{Dobson2006a} is missing the graphs given by Lemmas \ref{PSL213lem} and \ref{M23lem}.  Additionally, \cite[Theorem 3.2]{Dobson2006a} and \cite[Corollary 3.3]{Dobson2006a} are missing the group $\PSL(2,11)$ in its imprimitive action action on 55 points. Finally, \cite[Theorem 4.1(3)]{Dobson2006a} is missing these same graphs.

The result from \cite[Theorem 3.2(1)]{Dobson2006a} could be strengthened to digraphs by including the digraphs with simple and imprimitive automorphism groups.
\item The result \cite[Theorem]{LiWWX1994} does not consider the action of $\PSL(2,13)$ given in Lemma \ref{PSL213lem} nor the action of $M_{23}$ given in Lemma \ref{M23lem}.
\end{itemize}


\begin{thebibliography}{10}

\bibitem{AlspachP1982}
Brian Alspach and T.~D. Parsons, \emph{A construction for vertex-transitive
  graphs}, Canad. J. Math. \textbf{34} (1982), no.~2, 307--318. \MR{MR658968
  (84h:05063)}

\bibitem{Bloom1967}
David~M. Bloom, \emph{The subgroups of {${\rm PSL}(3,\,q)$} for odd {$q$}},
  Trans. Amer. Math. Soc. \textbf{127} (1967), 150--178. \MR{0214671}

\bibitem{magma}
Wieb Bosma, John Cannon, and Catherine Playoust, \emph{The {M}agma algebra
  system. {I}. {T}he user language}, J. Symbolic Comput. \textbf{24} (1997),
  no.~3-4, 235--265, Computational algebra and number theory (London, 1993).
  \MR{1484478}

\bibitem{Conway1985}
J.~H. Conway, R.~T. Curtis, S.~P. Norton, R.~A. Parker, and R.~A. Wilson,
  \emph{Atlas of finite groups}, Oxford University Press, Eynsham, 1985,
  Maximal subgroups and ordinary characters for simple groups, With
  computational assistance from J. G. Thackray. \MR{827219}

\bibitem{DixonM1996}
John~D. Dixon and Brian Mortimer, \emph{Permutation groups}, Graduate Texts in
  Mathematics, vol. 163, Springer-Verlag, New York, 1996. \MR{MR1409812
  (98m:20003)}

\bibitem{Dobson2006a}
Edward Dobson, \emph{Automorphism groups of metacirculant graphs of order a
  product of two distinct primes}, Combin. Probab. Comput. \textbf{15} (2006),
  no.~1-2, 105--130. \MR{MR2195578 (2006m:05108)}

\bibitem{DobsonHKM2017}
Edward Dobson, Ademir Hujdurovi\'c, Klavdija Kutnar, and Joy Morris, \emph{On
  color-preserving automorphisms of {C}ayley graphs of odd square-free order},
  J. Algebraic Combin. \textbf{45} (2017), no.~2, 407--422. \MR{3604062}

\bibitem{DobsonM2009}
Edward Dobson and Joy Morris, \emph{Automorphism groups of wreath product
  digraphs}, Electron. J. Combin. \textbf{16} (2009), no.~1, Research Paper 17,
  30. \MR{MR2475540}

\bibitem{Dobson2016}
Ted Dobson, \emph{On isomorphisms of {M}aru\v si\v c-{S}capellato graphs},
  Graphs Combin. \textbf{32} (2016), no.~3, 913--921. \MR{3489712}

\bibitem{KlinP1981}
M.~H. Klin and R.~P{\"o}schel, \emph{The {K}\"onig problem, the isomorphism
  problem for cyclic graphs and the method of {S}chur rings}, Algebraic methods
  in graph theory, {V}ol. {I}, {II} ({S}zeged, 1978), Colloq. Math. Soc.
  J\'anos Bolyai, vol.~25, North-Holland, Amsterdam, 1981, pp.~405--434.
  \MR{MR642055 (83h:05047)}

\bibitem{LiWWX1994}
Hui~Ling Li, Jie Wang, Lu~Yan Wang, and Ming~Yao Xu, \emph{Vertex primitive
  graphs of order containing a large prime factor}, Comm. Algebra \textbf{22}
  (1994), no.~9, 3449--3477. \MR{1278798}

\bibitem{LiebeckPS1988a}
Martin~W. Liebeck, Cheryl~E. Praeger, and Jan Saxl, \emph{On the {$2$}-closures
  of finite permutation groups}, J. London Math. Soc. (2) \textbf{37} (1988),
  no.~2, 241--252. \MR{928521}

\bibitem{LiebeckS1985a}
Martin~W. Liebeck and Jan Saxl, \emph{Primitive permutation groups containing
  an element of large prime order}, J. London Math. Soc. (2) \textbf{31}
  (1985), no.~2, 237--249. \MR{809945 (87c:20006)}

\bibitem{LuX2003}
Zai-Ping Lu and Ming-Yao Xu, \emph{On the normality of {C}ayley graphs of order
  {$pq$}}, Australas. J. Combin. \textbf{27} (2003), 81--93. \MR{1955389}

\bibitem{MarusicS1994}
D.~Maru{\v{s}}i{\v{c}} and R.~Scapellato, \emph{Classifying vertex-transitive
  graphs whose order is a product of two primes}, Combinatorica \textbf{14}
  (1994), no.~2, 187--201. \MR{MR1289072 (96a:05072)}

\bibitem{Marusic1988}
Dragan Maru{\v{s}}i{\v{c}}, \emph{On vertex-transitive graphs of order {$qp$}},
  J. Combin. Math. Combin. Comput. \textbf{4} (1988), 97--114. \MR{MR978524
  (89i:05142)}

\bibitem{MarusicP2002}
Dragan Maru{\v{s}}i{\v{c}} and Primo{\v{z}} Poto{\v{c}}nik, \emph{Classifying
  2-arc-transitive graphs of order a product of two primes}, Discrete Math.
  \textbf{244} (2002), no.~1-3, 331--338, Algebraic and topological methods in
  graph theory (Lake Bled, 1999). \MR{1844042}

\bibitem{MarusicS1992}
Dragan Maru{\v{s}}i{\v{c}} and Raffaele Scapellato, \emph{Characterizing
  vertex-transitive {$pq$}-graphs with an imprimitive automorphism subgroup},
  J. Graph Theory \textbf{16} (1992), no.~4, 375--387. \MR{MR1174460
  (93g:05066)}

\bibitem{MarusicS1993}
\bysame, \emph{Imprimitive representations of {${\rm SL}(2,2\sp k)$}}, J.
  Combin. Theory Ser. B \textbf{58} (1993), no.~1, 46--57. \MR{MR1214891
  (94a:20008)}

\bibitem{MarusicS1994a}
\bysame, \emph{Classification of vertex-transitive {$pq$}-digraphs}, Istit.
  Lombardo Accad. Sci. Lett. Rend. A \textbf{128} (1994), no.~1, 31--36 (1995).
  \MR{MR1434162 (98a:05078)}

\bibitem{PraegerWX1993}
Cheryl~E. Praeger, Ru~Ji Wang, and Ming~Yao Xu, \emph{Symmetric graphs of order
  a product of two distinct primes}, J. Combin. Theory Ser. B \textbf{58}
  (1993), no.~2, 299--318. \MR{MR1223702 (94j:05060)}

\bibitem{PraegerX1993}
Cheryl~E. Praeger and Ming~Yao Xu, \emph{Vertex-primitive graphs of order a
  product of two distinct primes}, J. Combin. Theory Ser. B \textbf{59} (1993),
  no.~2, 245--266. \MR{MR1244933 (94j:05061)}

\bibitem{Atlas}
et~al R.~Wilson, \emph{Atlas of finite group representations - version 3},
  http://brauer.maths.qmul.ac.uk/Atlas/v3/.

\bibitem{Sabidussi1958}
Gert Sabidussi, \emph{On a class of fixed-point-free graphs}, Proc. Amer. Math.
  Soc. \textbf{9} (1958), 800--804. \MR{MR0097068 (20 \#3548)}

\bibitem{WangX1993}
Ru~Ji Wang and Ming~Yao Xu, \emph{A classification of symmetric graphs of order
  {$3p$}}, J. Combin. Theory Ser. B \textbf{58} (1993), no.~2, 197--216.
  \MR{MR1223693 (94f:05071)}

\bibitem{WangFZWM2016}
Xiuyun Wang, Yanquan Feng, Jinxin Zhou, Jihui Wang, and Qiaoling Ma,
  \emph{Tetravalent half-arc-transitive graphs of order a product of three
  primes}, Discrete Math. \textbf{339} (2016), no.~5, 1566--1573. \MR{3475570}

\bibitem{Wielandt1964}
Helmut Wielandt, \emph{Finite permutation groups}, Translated from the German
  by R. Bercov, Academic Press, New York, 1964. \MR{MR0183775 (32 \#1252)}

\end{thebibliography}

\providecommand{\bysame}{\leavevmode\hbox to3em{\hrulefill}\thinspace}
\providecommand{\MR}{\relax\ifhmode\unskip\space\fi MR }
\providecommand{\MRhref}[2]{%
  \href{http://www.ams.org/mathscinet-getitem?mr=#1}{#2}
}
\providecommand{\href}[2]{#2}

\end{document}